\documentclass[11pt, a4paper]{amsart}
\usepackage{amsfonts}
\usepackage{amssymb,amsthm, amsmath}
\usepackage{bbm}
\usepackage[vmargin=2cm,hmargin=2cm]{geometry}
\usepackage[utf8]{inputenc}
\usepackage{xcolor}
\usepackage{pgf,tikz}
\usepackage{url}

\newtheorem{theorem}{Theorem}
\newtheorem{definition}[theorem]{Definition}

\newtheorem{corollary}[theorem]{Corollary}
\newtheorem{lemma}[theorem]{Lemma}

\theoremstyle{remark}

\newtheorem{example}[theorem]{Example}

\newtheorem{remark}[theorem]{Remark}

\title{Kronecker Square Roots and The Block vec Matrix}
\author{Ignacio Ojeda}
\address{Departamento de Matem\'aticas, Universidad de Extremadura,
E-06071 Badajoz, Espa\~na.} \email{ojedamc@unex.es}\thanks{The author is supported by the project MTM2012-36917-C03-01, National Plan I+D+I and by Junta de Extremadura (FEDER funds).}


\begin{document}

\begin{abstract}
Using the block vec matrix, I give a necessary and sufficient condition for factorization of a matrix into the Kronecker product of two other matrices. As a consequence, I obtain an elementary algorithmic procedure to decide whether a matrix has a square root for the Kronecker product.
\end{abstract}

\maketitle

\section*{Introduction}

My statistician colleague,  J.E. Chac\'on, asked me how to decide if a real given matrix $A$ has a square root for the Kronecker product (i.e., if there exists a $B$ such that $A = B \otimes B$) and, in the positive case, how to compute it. His questions were motivated by the fact that, provided that a certain real positive definite symmetric matrix has a Kronecker square root, explicit asymptotic expressions for certain estimator errors could be obtained. See \cite{Chacon}, for a discussion
of the importance of multivariate kernel density derivative estimation.

This note is written mostly due to the lack of a suitable reference for the existence of square roots for the Kronecker product, and it is organized as follows: first of all, I study the problem of the factorization of a matrix into a Kronecker product of two matrices, by giving a necessary and sufficient condition under which this happens (Theorem \ref{Th1}). As a preparation for the main result, I introduce the  \emph{block vec matrix} (Definition \ref{Def BV}). Now, the block vec matrix and Theorem \ref{Th1} solve our problem in a constructive way.

\section{Kronecker product factorization}

Throughout this note $\mathbb{N}$, $\mathbb{R}$, and $\mathbb{C}$ denote the sets of non-negative integers, real numbers, and complex numbers, respectively. All matrices considered here have real or complex entries;
$A^\top$ denotes the transpose of $A$ and $\mathrm{tr}(A)$ denotes its trace.

The operator that transforms a matrix into a stacked vector is known as the \emph{vec operator} (see, \cite[Definition 4.2.9]{HJ91} or \cite[$\S$ 7.5]{Schott}). If $A = \big( \mathbf{a}_1 \vert \ldots \vert \mathbf{a}_n \big)$ is an $m \times n$ matrix whose columns are $\mathbf{a}_1, \ldots, \mathbf{a}_n$, then $\mathrm{vec}(A)$ is the $m n \times 1$ matrix
$$\mathrm{vec}(A) = \left(\begin{array}{c} \mathbf{a}_1 \\ \vdots \\ \mathbf{a}_n \end{array}\right).$$

The following definition generalizes the vec operation and is the key to all that follows.

\begin{definition}\label{Def BV}
Let $A = (A_{ij})$ be an $mp \times nq$ matrix partitioned into block matrices $A_{ij}$, each of order $p \times q$. The \emph{block vec matrix} of $A$ corresponding to the given partition is the $m n \times p q$ matrix $$\mathrm{vec}^{(p \times q)}(A) = \left(\begin{array}{c} A_1 \\ \vdots \\ A_n \end{array}\right),\quad \text{where each}\ A_j = \left(\begin{array}{c} \mathrm{vec}(A_{1j})^\top \\ \vdots \\ \mathrm{vec}(A_{mj})^\top \end{array}\right).$$
\end{definition}

If $A$ is $m \times n$, it is instructive to verify the following identities corresponding to four natural ways to partition it:
\begin{itemize}
\item $p=q=1: \mathrm{vec}^{(1 \times 1)}(A) = \mathrm{vec}(A)$,
\item $p=m, q=n: \mathrm{vec}^{(m \times n)}(A) = \mathrm{vec}(A)^\top$,
\item $p=1, q=n$ (partition by rows)$: \mathrm{vec}^{(1 \times n)}(A) = A$,
\item $p=m, q=1$ (partition by columns)$: \mathrm{vec}^{(m \times 1)}(A) = A^\top$.
\end{itemize}
If $A$ is $mp \times nq$ and $\mathrm{vec}(A)$ is partitioned into $nq$ blocks, each of size $mp \times 1$, then a computation reveals that $\mathrm{vec}^{(mp \times 1)}(\mathrm{vec}(A)) = A^\top$.

Let $B = (b_{ij})$ be an $m \times n$ matrix and let $C$ be a $p \times q$ matrix. The \emph{Kronecker product} of $B$ and $C$, denoted by $B \otimes C$, is the $mp \times nq$ matrix $$B \otimes C = \left(\begin{array}{cccc} b_{11} C & b_{12} C & \cdots & b_{1n} C \\ b_{21} C & b_{22} C & \cdots & b_{2n} C \\ \vdots & \vdots & & \vdots \\ b_{m1} C& b_{m2} C & \ldots & b_{mn} C \end{array}\right).$$  

The following result is straightforward; see \cite[Theorem 1]{VP93}.

\begin{lemma}\label{Lemma1}
Let $B$ be an $m \times n$ matrix and let $C$ be a $p \times q$ matrix. Then $$\mathrm{vec}^{(p \times q)}(B \otimes C) = \mathrm{vec}(B) \mathrm{vec}(C)^\top.$$ In particular, $\mathrm{vec}^{(p \times q)}(B \otimes C)^\top = \mathrm{vec}^{(m \times n)}(C \otimes B)$.
\end{lemma}

It may be (but need not) be possible to factor a given matrix, suitably partitioned, as a Kronecker product of two other matrices. For example, a zero matrix can always be factored as a Kronecker product of a zero matrix and any matrix of suitable size. The following theorem provides a necessary and sufficient condition for a Kronecker factorization.

\begin{theorem}\label{Th1}
Let $A = (A_{ij})$ be a nonzero $mp \times nq$ matrix, partitioned into blocks of order $p \times q$. There exist matrices $B$ (of order $m \times n$) and $C$ (of order $p \times q$) such that $A = B \otimes C$ if and only if $\mathrm{rank}(\mathrm{vec}^{(p \times q)}(A)) = 1$.
\end{theorem}

\begin{proof}
If $A = B \otimes C$ is a factorization of the stated form, then $B, C, \mathrm{vec}(B)$, and $\mathrm{vec}(C)$ must all be nonzero. Lemma \ref{Lemma1} ensures that $$\mathrm{rank}(\mathrm{vec}^{(p \times q)}(A)) = \mathrm{rank}(\mathrm{vec}^{(p \times q)}(B \otimes C)) = \mathrm{rank}(\mathrm{vec}(B) \mathrm{vec}(C)^\top) = 1.$$
Conversely, since $A \neq 0$, there are indices $r$ and $s$ such that $A_{rs} \neq 0$ and hence $\mathrm{vec}(A_{rs}) \neq 0$. Since $\mathrm{rank}(\mathrm{vec}^{(p \times q)}(A)) = 1$, each row of $\mathrm{vec}^{(p \times q)}(A)$ is a scalar multiple of any nonzero row. Thus, there are scalars $b_{ij}$ such that each $\mathrm{vec}(A_{ij}) = b_{ij}\, \mathrm{vec}(A_{rs})$. This means that $A = B \otimes C$, in which $B = (b_{ij})$ and $C = A_{rs}$.
\end{proof}

Notice that the preceding proof provides a simple construction for a pair of Kronecker factors for $A$ if $\mathrm{rank}(\mathrm{vec}^{(p \times q)}(A)) = 1$.

The block vec matrix can be used to detect not only whether a given matrix has a Kronecker factorization of a given form, but also, if it does not, how closely it can be approximated in the Frobenius norm by a Kronecker product. A best approximation is determined by the singular value decomposition of the block vec matrix. For details, see \cite{VP93}, where the block vec matrix is called the \emph{rearrangement matrix}.

\begin{example}\label{ex1}
Consider $$A =  \left(\begin{array}{cc} 2 & 1 \\ 2 & 0 \\ \cline{1-2} 3 & 0 \\ 0 & 3 \end{array}\right).$$ Since $$\mathrm{vec}^{(2 \times 2)}(A) = \left(\begin{array}{cccc} 2 & 2 & 1 & 0 \\ 3 & 0 & 0 & 3 \end{array}\right)$$ has rank $2,$ Theorem \ref{Th1} ensures that $A \neq B \otimes C$, for any $B$ and $C$ of order $2 \times 1$ and $2 \times 2$, respectively.
\end{example}

The set of matrices that factorices into the Kronecker product of two other matrices have the following interpretation in Algebraic Geometry.

\begin{remark}
Let $A$ be a real $mp \times nq$ matrix and consider the big matrix 
$$ M = \left(\begin{array}{c} I_{mn} \otimes \mathbf{u}_{pq}^\top \\ \mathbf{u}_{mn}^\top \otimes I_{pq} \end{array}\right)
$$
where $\mathbf{u}_\bullet$ is the all-ones vector of dimension $\bullet$. Let $K$ be equal to $\mathbb{R}$ or $\mathbb{C}$ and let $$\varphi_M : K[X_{11}, \ldots, X_{m1}, X_{12}, \ldots, X_{mn\, pq}] \longrightarrow K[t_1, \ldots, t_{mn+pq}]$$ be the $K-$algebra map such that $\varphi(\mathbf{X}^\mathbf{u}) = \mathbf{t}^{M \mathbf{u}}$, with $\mathbf{X}^\mathbf{u} := X_{11}^{u_1} \cdots X_{mn\, pq}^{u_{mn pq}}$ and $\mathbf{t}^\mathbf{v} := t_1^{v_1} \cdots t_{mn+pq}^{v_{mn+pq}}$. Then $X = (x_{ij})$ has rank $1$ if and only if $\mathrm{vec}(X)^\top$ is a zero of $\ker \varphi_M$ (see, e.g. \cite[Section 2]{GMS}). Therefore, by Theorem \ref{Th1}, the set of $mn \times pq$ matrices which factor into the Kronecker product of an $m \times n$ matrix and a $p \times q$ matrix is the following \emph{algebraic set} 
$$\mathcal{V}(\ker \varphi_M) = \Big\{ A \in K^{mn \times pq} \mid \mathrm{vec}(\mathrm{vec}^{(p \times q)}(A))^\top \in \ker \varphi_M \Big\}.$$

In Algebraic Statistics, $\ker \varphi_M$ is the ideal associated to two independent random variables with values in $\{1, \ldots, mn\}$ and $\{1, \ldots, pq\}$. Thus, as it is well know in Statistics, \emph{factorization means independence and vice versa}.
\end{remark}

We can now give a solution to the problem that motivates this paper.

\begin{corollary}
If $A$ is a nonzero $m^2 \times n^2$ matrix and $B$ an $n \times n$ matrix, then
\begin{itemize}
\item[(a)] $A = B \otimes B$ if and only if $\mathrm{vec}^{(m \times n)}(A) = \mathrm{vec}(B) \mathrm{vec}(B)^\top$,
\item[(b)] if $A = B \otimes B$, then $\mathrm{vec}^{(m \times n)}(A)$ is symmetric and has rank one.
\end{itemize}
\end{corollary}

Is the necessary condition in the preceding corollary sufficient? It is not surprising that the answer depends on the field. For example, if $m = n = 1$ and $A = (-1)$, then $\mathrm{vec}^{(1 \times 1)}(A) = (-1)$ is symmetric and has rank one, but $A$ has no real Kronecker square root; it does have complex Kronecker square roots $B = (\pm \mathrm{i})$, but these are the only ones.

\begin{theorem}
Let $A$ be an $m^2 \times n^2$ real or complex matrix. Suppose that $\mathrm{vec}^{(m \times n)}(A)$ is symmetric and has rank one. 
\begin{itemize}
\item[(a)] There is an $m \times n$ matrix $B$ such that $A = B \otimes B$.
\item[(b)] If $B$ and $C$ are $m \times n$ matrices such that $A = B \otimes B = C \otimes C$, then $C = \pm B$.
\item[(c)] If $A$ is real, there is a real $m \times n$ matrix $B$ such that $A = B \otimes B$ if and only if $\mathrm{tr}(\mathrm{vec}^{(m \times n)}(A)) > 0$.
\end{itemize}
\end{theorem}

\begin{proof}
Any complex symmetric matrix has a special singular value decomposition, that is unique in a certain way; see \cite[Corollary 4.4.4: Autonne's theorem]{HJ13}. In the case of a rank one symmetric matrix $Z$ whose largest (indeed, only nonzero) singular value is $\sigma$, Autonne's theorem says that there is a unit vector $\mathbf{u}$ such that $Z = \sigma \mathbf{u} \mathbf{u}^\top$. Moreover, if $\mathbf{v}$ is a unit vector such that $Z = \sigma \mathbf{v} \mathbf{v}^\top$, then $\mathbf{v} = \pm \mathbf{u}$. If we use Autonne's theorem to represent the block vec matrix as $\mathrm{vec}^{(m \times n)}(A) = \sigma \mathbf{u} \mathbf{u}^\top = (\sigma^{1/2} \mathbf{u}) (\sigma^{1/2} \mathbf{u})^\top$ and define $B$ by $\mathrm{vec}(B) = (\sigma^{1/2} \mathbf{u})$, we have $\mathrm{vec}^{(m \times n)}(A) = \mathrm{vec}(B) \mathrm{vec}(B)^\top$. The preceding corollary now ensures that $A = B \otimes B$ and the assertion in (b) follows from the uniqueness part of Autonne's theorem.

Now suppose that $A$ is real. If there is a real $B$ such that $A = B \otimes B$, then $\mathrm{tr}(\mathrm{vec}^{(m \times n)}(A) = \mathrm{tr}(\mathrm{vec}(B) \mathrm{vec}(B)^\top) = \mathrm{vec}(B)^\top B$ is positive since it is the square of the Euclidean norm of the (necessarily nonzero) real vector $\mathrm{vec}(B)$. Conversely, the spectral theorem ensures that any real symmetric matrix can be represented as $Q \Lambda Q^\top$, in which $Q$ is real orthogonal and $\Lambda$ is real diagonal. Since the block vec matrix is real symmetric and has rank one, we can take $\Lambda = \mathrm{diag}( \lambda, 0, \ldots, 0)$ and represent $\mathrm{vec}^{(m \times n)}(A) = \lambda \mathbf{q} \mathbf{q}^\top$ in which $\mathbf{q}$ is the first column of $Q$. If $\lambda = \mathrm{tr}(\mathrm{vec}^{(m \times n)}(A)) > 0$, then $\mathrm{vec}^{(m \times n)}(A) = (\lambda^{1/2} \mathbf{q}) (\lambda^{1/2} \mathbf{q})^\top = \mathrm{vec}(B) \mathrm{vec}(B)^\top,$ in which $\mathrm{vec}(B) = \lambda^{1/2} \mathbf{q}$ (and hence also $B$) is real.
\end{proof}

The uniqueness part of the preceding theorem has some perhaps surprising consequences.

\begin{corollary}
Let $A$ be a nonzero $m^2 \times n^2$ real or complex matrix and suppose that $A = B \otimes B$ for some $m \times n$ matrix $B$.
\begin{itemize}
\item[(a)] $A$ is symmetric if and only if $B$ is either symmetric or skew symmetric.
\item[(b)] $A$ is not skew symmetric.
\item[(c)] $A$ is Hermitian if and only if $B$ is either Hermitian or skew Hermitian.
\item[(d)] $A$ is Hermitian positive definite if and only if $B$ is Hermitian and definite (positive or negative).
\item[(e)] $A$ is skew Hermitian if and only if $\mathrm{e}^{\mathrm{i}\, \pi /4} B$ is Hermitian.
\item[(f)] $A$ is unitary if and only if $B$ is unitary.
\item[(g)] If $B$ is real, then $A$ is real orthogonal if and only if $B$ is real orthogonal.
\item[(h)] $A$ is complex orthogonal if and only if either $B$ or $\mathrm{i} B$ is complex orthogonal.
\end{itemize}
\end{corollary}

\begin{proof}
(a) $A^\top = B^\top \otimes B^\top$, so $A = A^\top$ if and only if $A = B \otimes B = B^\top \otimes B^\top$, which holds if and only if $B^\top = \pm B$. 

(b) If $A^\top = -A^\top,$ then $-A = - B \otimes B = (\mathrm{i} B) \otimes (\mathrm{i} B) = B^\top \otimes B^\top$ and hence $B^\top = \pm \mathrm{i} B = \pm \mathrm{i} (B^\top)^\top = \pm \mathrm{i} (\pm \mathrm{i} B)^\top = - B^\top$, so $B = 0$. 

(c) $A = A^*$ if and only if $A = B \otimes B = B^* \otimes B^*$, that is to say, $B^* = \pm B$. 

(d) Using (c) and the fact that the eigenvalues of $B \otimes B$ are the pairwise products of the eigenvalues of $B$, we can exclude the possibility that $B$ is skew Hermitian since in that case its nonzero eigenvalues (there must be at least one) would be pure imaginary and hence $B \otimes B$ would have at least one negative eigenvalue. 

(e) Under our hypothesis, the following statements are equivalent
(i) $A = - A^*$, (ii) $A = B \otimes B = - B^* \otimes B^* = (i B)^* \otimes (i B)^*$ (iii) $B^* = \pm \mathrm{i} B$ (iv) $(\mathrm{e}^{\mathrm{i}\, \pi /4} B)^* = \pm \mathrm{e}^{\mathrm{i}\, \pi /4} B$, and so the claim follows.

(f) $A^{-1} = B^{-1} \otimes B^{-1} = B^* \otimes B^*$ if and only if $B^* = \pm B^{-1}$ or, that is the same, $B B^* = \pm I$. However $B B^* = -I$ is not possible since $B B^*$ is positive definite. 

(g) Follows from (f). 

(h) $A^{-1} = B^{-1} \otimes B^{-1} = B^\top \otimes B^\top$ if and only if $B^\top = \pm B^{-1}$ that is to say $B B^\top = \pm I$ or, equivalently,  either $B B^\top = I$ or $(\mathrm{i}B) (\mathrm{i} B)^\top = I$.
\end{proof}

\section*{Acknowledgment.}
I would like to thank J.E Chac\'on for pointing me towards the question of the square roots of matrices for the Kronecker product. I also thank the anonymous referee for the comments and useful suggestions.


\end{document}